\documentclass[12pt]{amsart}

\usepackage{a4}
\usepackage{etex}
\usepackage[T1]{fontenc} 
\usepackage[utf8]{inputenc}
\usepackage[english]{babel}

\usepackage{etoolbox} 
\usepackage[normalem]{ulem} 
\usepackage{comment}
\usepackage[autostyle]{csquotes}

\usepackage{nicefrac} 
\usepackage[centercolon=false]{mathtools} 
\usepackage[overload, ntheorem]{empheq}  
\usepackage[all]{xy}    

\input{mathcoms}
\newcommand{\calM}{\mathcal{M}}
\newcommand{\FM}{\mathrm{FM}}
\newcommand{\Exp}{\mathbb{E}}
\newcommand{\AP}{\mathrm{AP}}

\newcommand{\prfnoi}{\smallskip\noindent}

\newcommand{\suchthat}{\, |\,}
\newcommand{\bsuchthat}{\, \big|\,}

\newcommand{\dps}{\displaystyle}

\newcommand{\Pot}{\mathcal{P}}
\newcommand{\fin}{\mathrm{fin}}
\newcommand{\bnorm}[2][\relax]{
   \ifx#1\relax \ensuremath{\bigl\Vert#2\bigr\Vert}
   \else \ensuremath{\bigl\Vert#2\bigr\Vert_{#1}}
   \fi}
\newcommand{\Bnorm}[2][\relax]{
   \ifx#1\relax \ensuremath{\Bigl\Vert#2\Bigr\Vert}
   \else \ensuremath{\Bigl\Vert#2\Bigr\Vert_{#1}}
   \fi}
\newcommand{\mix}{\operatorname{mix}}
\newcommand{\Bv}[1]{\left\llbracket\,#1\,\right\rrbracket}

\newcommand{\Clop}{\mathrm{Clop}}

\DeclareMathAlphabet{\mathpzc}{OT1}{pzc}{m}{it}

\newcommand{\Om}{\Omega}

\newcommand{\emdf}{\bf}
\newcommand{\vanish}[1]{\relax}
\newcommand{\abs}[1]{\vert #1 \vert}           
\newcommand{\norm}[1]{\Vert #1 \Vert}           
\newcommand{\car}{\mathbf{1}}
\newcommand{\Ce}{\mathrm{C}}

\newcommand{\veps}{\varepsilon}
\newcommand{\set}[1]{\left[\,#1\,\right]}

\newcommand{\sprod}[2]{(#1|#2)}

\newcommand{\konj}[1]{\overline{#1}}
\newcommand{\leer}{\emptyset}

\newcommand{\Ball}{\mathrm{B}}

\newcommand{\BL}{\mathscr{L}}

\newcommand{\calF}{\mathcal{F}}

\newcommand{\tensor}{\otimes}
\newcommand{\Ell}[1]{\mathrm{L}^{#1}}

\newcommand{\B}{\mathbb{B}}
\newcommand{\dann}{\Rightarrow}
\newcommand{\gdw}{\Leftrightarrow}

\newcommand{\ocl}{\mathrm{ocl}}

\renewcommand{\ds}{\mathrm{ds}}
\newcommand{\wm}{\mathrm{wm}}

\renewcommand{\A}{\mathbb{A}}
\renewcommand{\B}{\mathbb{B}}
\newcommand{\cl}{\mathrm{cl}}

\newcommand{\scrE}{\mathscr{E}}

\newcommand{\prX}{\mathrm{X}}
\newcommand{\prY}{\mathrm{Y}}

\usepackage{marginnote}

\usepackage{mathrsfs} 
\usepackage{euscript} 
\usepackage{stmaryrd}
\usepackage{textgreek}
\usepackage{anyfontsize} 
\usepackage{bbm} 

\numberwithin{equation}{section}

\theoremstyle{plain}
  \newtheorem{theorem}{Theorem}[section]
  \newtheorem{lemma}[theorem]{Lemma}
\newtheorem{corollary}[theorem]{Corollary}
\newtheorem{proposition}[theorem]{Proposition}

\theoremstyle{definition}
\newtheorem{definition}[theorem]{Definition}
\newtheorem{remark}[theorem]{Remark}
\newtheorem{example}[theorem]{Example}

\allowdisplaybreaks[2]

\vanish{
\def\mystepname{Step}

\def\mytocname{Table of Contents}

}

\usepackage{enumerate} 
\usepackage[shortlabels]{enumitem} 
\setlist[enumerate]{
  topsep = 1ex,  
  itemsep = 1ex, 
  partopsep = 1ex,
  parsep = 0ex,
  leftmargin = *,
  labelsep = 0.2cm,
  label = (\alph{enumi}),
  font = \normalfont,
  ref = \labelenumi
}
\setlist[itemize]{
  partopsep = 1ex,
  parsep = 1ex
}
\newenvironment{aufzi}{\begin{enumerate}[(a)]}{\end{enumerate}}
\newenvironment{aufzii}{\begin{enumerate}[(i)]}{\end{enumerate}}

\newlist{aufziv}{enumerate}{1}
\setlist*[aufziv]{label= \upshape{(\arabic*)},
  leftmargin=0pt,itemindent=2em, itemsep=0.7ex}

\author{Markus Haase}

\address{Markus Haase, Mathematisches Seminar, Christian-Albrechts-Universität zu Kiel, Heinrich-Hecht-Platz 6, 24118 Kiel, Germany}
\email{haase@math.uni-kiel.de}

\author{Henrik Kreidler}

\address{Henrik Kreidler, Fachgruppe Mathematik und Informatik, Bergische Universität Wuppertal, Gaußstraße 20,
42119 Wuppertal, Germany}
\email{kreidler@uni-wuppertal.de}

\title[Precompactness in Kaplansky--Hilbert modules]{Precompactness
  notions in Kaplansky--Hilbert modules and extensions with discrete spectrum}

\date{9 February 2026}

\makeatletter
\@namedef{subjclassname@2020}{\textup{2020} Mathematics Subject Classification}
\makeatother


\usepackage[backref       = false, 
                breaklinks    = true, hidelinks]{hyperref}

\begin{document}
\renewcommand{\contentsname}{\mytocname}


\begin{abstract}
This paper is a continuation of our work on the functional-analytic core of the
classical Furstenberg–Zimmer theory.  We introduce and study (in the
framework of lattice-ordered spaces) 
the notions of total order-boundedness and 
uniform total order-boundedness. Either one generalizes the concept
of ordinary precompactness known from
 metric space theory. These new notions are then used to define and
 characterize ``compact extensions'' of general measure-preserving
 systems (with no restrictions on the underlying probability 
spaces nor on the acting groups). In particular, it is (re)proved that
compact extensions 
and extensions with discrete spectrum are one and the same
thing. Finally, we show that 
under natural hypotheses a subset of a Kaplansky–Banach module is
totally order bounded if 
and only if it is cyclically compact (in the sense of Kusraev).
\end{abstract}

\maketitle

\section{Introduction}

The Furstenberg--Zimmer structure theorem is one of the 
great  milestones in structural ergodic theory. It goes back
to the seminal works of Zimmer \cite{Zimm1976,Zimm1976b} and Furstenberg \cite{Furs1977}. 
The heart of the
matter is a diligent examination of {\em extensions} of
measure-preserving systems and a fundamental {\em dichotomy}: either
an extension is ``weakly mixing'' or it has an in some sense
``well-structured'' intermediate extension.
Zimmer showed that ``well-structured'' could mean ``with relative
discrete spectrum'' or, equivalently, ``isometric''. Furstenberg
recovered Zimmer's results and added ``compact'' in a later paper
\cite{FuKa1978}, see also \cite{Furstenberg1981}.

Both authors relied on  technical assumptions on the underlying
probability space (a standard Lebesgue space) and on the acting group
(Zimmer: standard Borel group, Furstenberg: $\Z^d$). After the strong
popularization of structural ergodic theory in the aftermath of the
Fields medals for Gowers (1998) and Tao (2004), there was a growing
interest in freeing the result from the abovementioned restrictions and to 
gain an understanding of it in terms of more abstract structures. 

Relatively recently, this has been achieved by several different (but
related) approaches. Our own one, in a work together with N. Edeko,
makes use of  {\em Kaplansky--Hilbert
  modules} (the theory which had to a large extent to be developed for this
purpose), cf.\,\cite{EHK2024}. This approach
differs from others in that the main
dichotomy for extensions of measure-preserving systems
is reduced to an abstract,  purely functional-analytic statement
about so-called {\em KH-dynamical systems}.  In particular,
the Hilbert space structure of the
surrounding $\Ell{2}$-space is not employed.

However, the dichotomy proved in \cite{EHK2024} involves only 
the notion of ``discrete spectrum'' for the structured part of the
extension. And it remained  open how to define ``compact'' extensions in 
our abstract setting and to prove that compact extensions are 
the same as extensions with discrete spectrum. It is the purpose
of this paper (announced at the end of \cite{EHK2024}) to fill this gap.

\medskip
The paper is organized as follows. In Section \ref{s.tob} we first recall the
notion
of a lattice-normed space (but refer to Appendix \ref{s.app} for 
all the other related notions like order-convergence, Stone algebra 
or Kaplansky--Hilbert module). Then we define {\em totally
  order-bounded} and {\em uniformly totally order-bounded} subsets of
such spaces as generalizations of ordinary precompactness. Moreover,
we  establish a couple of relevant technical properties (Lemma
\ref{tob.l.properties}). The main result is that a bounded subset of a finite-rank
KH-module  is uniformly totally order-bounded (Proposition
\ref{tob.p.finrank}). Finally, 
following Tao's definition in \cite{Tao2009} an
alternative
compactness notion involving {\em zonotopes} is considered and equivalence 
with uniform total order-boundedness is proved.

In Section \ref{s.cpe} we recall from \cite{EHK2024} the relevant
notions from the theory of (abstract)
measure-preserving systems and their extensions, in particular
extensions with discrete spectrum. Next, we introduce {\em compact
extensions} (employing the abstract precompactness notions from
Section \ref{s.tob}). The main results are then Theorem
\ref{cpe.t.cpds} and its Corollary \ref{cpe.c.cpds}, in which
compact extensions are characterized in several ways, and  in
particular as extensions with discrete spectrum. The
proof of this characterization, however, does not rely only on 
results from abstract KH-dynamical systems theory, but involves the  
surrounding Hilbert space structure in a crucial way. 
Whether this can be avoided remains open, see Remark
\ref{cpe.r.cpe-KH-char?}.

In the last Section \ref{s.rcc}, we relate the precompactness notions
from Section \ref{s.tob} with the notion of {\em cyclical compactness}
appearing in the works of Kusraev \cite{Kusr2000} and mentioned 
at several places in \cite{EHK2024}.
As a main result we find that
under natural assumptions a mix-complete subset of a lattice-normed space
is totally order-bounded if and only if it is relatively cyclically
compact (Proposition \ref{rcc.p.tob-rcc}).

\medskip
Let us make some remarks concerning originality. 
As mentioned above, the equality of compact extensions and 
extensions of discrete spectrum (in Corollary \ref{cpe.c.cpds}) is classical under the usual assumptions in standard ergodic theory.
In the framework of ``abstract''
measure-preserving systems---and hence  without restrictions on the
underlying probability space nor on the acting group---%
it has been proved by Jamneshan in
\cite[Theorem 4.1]{Jam2023}. Moreover, also  the connections to
notions from conditional set theory and Boolean-valued analysis have
been 
pointed out at several occasions in Jamneshan's work, see
\cite[Remarks 3,5, 4.3, 4.5, 4.7, and 6.4]{Jam2023}. 

Our contribution in this article regards less the results itself but
rather  the theoretical framework to obtain them. 
It emphasizes on the purely functional-analytic essence of (the
largest part) of the results and is another step  
towards explaining  how the plethora of concepts/results for extensions 
appearing in the ergodic theoretic literature can be stated/proved
in purely functional-analytic terms. Where in the usual approach
to extensions of measure--preserving systems arguments  
are motivated by {\em analogies}  with classical concepts
or results from Banach or Hilbert space theory, we show that there is
a purely functional-analytic theory that renders the results
actual {\em generalizations}. 

We also hope that our work (together
with \cite{EHK2024} and the upcoming article \cite{HaaKrepre} on Furstenberg's main
theorem) will  make the powerful Furstenberg--Zimmer theory more accessible, in
particular for readers with a functional-analytic background.

\subsection*{Acknowledgements}
The authors are grateful for inspiring discussions with Nikolai Edeko
and Asgar Jamneshan, and for helpful comments of the anonymous
referee. Both authors also acknowledge the financial
support from the DFG (Henrik Kreidler: project number  451698284;  Markus Haase:
project number 431663331).

\section{Total Order-Boundedness}\label{s.tob}

In this section we generalize in several ways
the notion of total boundedness (=precompactness) for subsets
of normed spaces to subsets of
lattice-normed spaces. To this aim, we 
fix (once and for all) 
a commutative unital $C^*$-algebra $\A$
and consider it as a Banach lattice in the
canonical way.  

A {\emdf lattice-normed space} over $\A$ is a vector space $E$ together with a
mapping  $\abs{\cdot}\colon E \to \A_+$ with the following properties:
\[  \abs{x}=0 \gdw x=0,\quad \abs{\lambda x} = \abs{\lambda} \abs{x},\quad \abs{x+y}\le \abs{x} + \abs{y} \qquad  (x,y\in E,\, \lambda\in \C).
\]
The {\emdf closed ball} and the {\emdf open ball}
in this norm with center $x\in E$ and radius 
$t\in\A_+ $ are the sets
\[  \Ball_E[x;t] := \{ y\in E \suchthat \abs{x- y}\le t\} \quad 
\text{and}\quad  
\Ball_E(x;t) := \{ y\in E \suchthat \abs{x- y} < t\},
\]
respectively. 

Each lattice-normed space carries a
natural norm given by $\norm{x}_E := \norm{ \abs{x} }_\A$ for $x\in
E$. Given $r \in \R_{\ge 0}$ one has $\abs{x} \le r\car$ iff $\norm{x}_E
\le r$ and hence  
\[   \Ball_E[x; r\car] = \{ y\in E \suchthat \abs{y - x}\le r\} = 
\{ y \in E \suchthat \norm{y-x}_E \le r\} = \Ball_E[x;r]
\]
where the latter set is the closed ball with respect to the norm.  

A lattice-normed space over $\A= \C$ is nothing
but a normed space, and in this case $\norm{\cdot}_E$ coincides with
the original norm. So, a lattice-normed space 
is an analogue of a  normed space, but with elements of 
$\A_+$ taking the roles of the possible values
for the ``norm''. If $E$ is even 
a lattice-normed {\em module}, then elements
of $\A$ take the role of the scalars in all  
respects, and the analogy is even more striking.

\medskip
The elementary theory of lattice-normed spaces
is strongly analogous to the theory 
of normed spaces, but with 
the usual notions of convergence, closedness, continuity and
completeness suitably modified
(namely: {\em order-convergence}, {\em order-closedness}, {\em
  order-continuity}, and {\em order-completeness}). This theory
reduces to the classical one in the case $\A= \C$, and is
only slightly more difficult in the general case.
However, as we certainly cannot speak of common
knowledge here, we have collected basic definitions
and some elementary statements in Appendix \ref{s.app}. For proofs and
further information we refer to \cite[Sec.{ }1]{EHK2024}. The
presentation there centers around
 {\em Kaplansky--Hilbert modules},
which  are those instances of lattice-normed
spaces most relevant for us.

\medskip

Let us now turn to the main topic of this section.

\smallskip

\begin{definition}\label{tob.d.tob}
	Let $E$ be a lattice-normed space over a unital commutative C*-algebra $\mathbb{A}$. A subset $M \subset E$ is \textbf{totally order-bounded} or 
\textbf{order-precompact} if 
there is a net $(u_\alpha)_{\alpha}$ in $\mathbb{A}_+$  decreasing to
zero and such that for every $\alpha$ there is a finite set
$F\subseteq E$ with
		\begin{align*}
			\inf_{y\in F} |x-y| \leq u_\alpha \quad \textrm{ for every } x \in M.
		\end{align*}
And $M$ is \textbf{uniformly totally order-bounded} if for every
$\veps \in \R_{> 0}$ there is a finite set $F\subseteq E$ such that
		\begin{align*}
			\inf_{y\in F} |x-y| \leq \veps \car
                  \quad \textrm{ for every } x \in M.
		\end{align*}
\end{definition}

\smallskip

\begin{remark}\label{tob.r.tob}
\begin{aufziv}
\item Uniform total boundedness 
is related to the notion of relative uniform convergence in vector lattices, see \cite[Sec.\,1.3.4]{Kusr2000}.
It generalizes in 
the most straightforward way
ordinary precompactness in
  a normed space to subsets of lattice-normed spaces,
namely by replacing the ordinary norm by the 
lattice-norm. 

In the context of extensions $\prX|\prY$ of
probability spaces, uniformly totally order-bounded 
subsets of $\Ell{2}(\prX|\prY)$ are sometimes called
{\em conditionally precompact}. See also 
Remark \ref{cpe.r.cpc} below.

\item Observe the validity of the implications
\[ \text{$M$ unif.{}  totally order-bdd}\quad \dann\quad
\text{$M$ totally order-bdd}\quad\dann\quad \text{$M$ (order-)bdd}.
\]
The first implication is trivial; the second one follows from
\[  \abs{x} \le \inf_{y\in F} \abs{x-y} + \sup_{y\in F} \abs{y} \qquad
\text{for all $x\in E$ and all finite $F\subseteq E$.}
\]
Evidently, if $\A = \C\car$ 
then 
``totally order-bdd = uniformly order-bdd = precompact''. 
\end{aufziv}
\end{remark}

\medskip

From now on we suppose that {\em $\A$ is 
order-complete as a lattice-ordered space
over itself}. That is, $\A$ is a
{\emdf Stone algebra} (see Appendix \ref{s.app}).
Then we may  rephrase the definitions above 
in terms 
of (order-)convergence. To this aim, consider
the set 
\[ \Pot_{\fin}(E) = \{ F \subseteq E \suchthat \text{$F$ finite}\}
\]
of all finite subsets of $E$ to be upwards directed by
inclusion. To  a given bounded subset $M\subseteq E$ we
associate the decreasing net
\[  F \mapsto \sup_{x\in M} \inf_{y\in F} \abs{x-y}
\qquad(F\in \Pot_{\fin}(E)).
\]
Then $M$ is totally order-bounded if and only if 
			\begin{align*}		
\inf_{F \in \Pot_\fin(E)} \sup_{x \in M} \inf_{y \in F} |x-y| =0,
			\end{align*}
which means that the said net order-converges/decreases to zero in $\A_+$. 
Similarly, $M$ is uniformly totally order-bounded if and only if 
			\begin{align*}
				\inf_{F \in \Pot_\fin(E)}\Bnorm{\sup_{x \in M} \inf_{y \in F} |x-y|}_{\A} =0,
			\end{align*}
		i.e., if the net $F \mapsto \sup_{x \in M} \inf_{y \in F} |x-y|$ norm-converges to zero in $\A$.

\medskip

As mentioned, every uniformly totally order-bounded subset is totally
order-bounded. The following example shows that the converse is false,
even for subsets of Kaplansky--Hilbert modules.

\smallskip
\begin{example}
Let $H$ be an infinite-dimensional Hilbert space with
norm $\norm{\cdot}_H$. The space $E := \ell^\infty(\N; H)$
of bounded $H$-valued sequences is a KH-module over $\A :=
\ell^\infty= \ell^\infty(\N; \C)$, the space of all bounded scalar
sequences. The lattice-valued norm is 
\[   \abs{f} := ( n \mapsto \norm{f(n)}_H).
\]
Let $(e_n)_{n \in \N}$ be an orthonormal system in $H$. We claim that the set
\[ M := \{ \car_{\{k\}} \tensor e_j \suchthat k,j\in \N,\, 1\le j \le k\}
\]
is totally order-bounded, but not uniformly totally 
order-bounded in $E$. (We write $f\tensor e$ for the function
$(n\mapsto f(n)e)\colon \N \to H$ whenever $f\in \ell^\infty$ and $e\in H$.)

To prove the first claim, we take $n \in \N$ and let 
\[ F_n := \{0\} \cup \{ \car \tensor e_l \suchthat 1\le l \le n\}.
\]
Fix $1 \le j \le  k$. Then, since $g = 0 \in F_n$,
\[ \inf_{g\in F_n} \abs{ \car_{\{k\}} \tensor e_j - g} 
\le \car_{\{k\}}.
\]
If, in addition, $j \le n$, then 
\[ \inf_{g\in F_n} \abs{ \car_{\{k\}} \tensor e_j - g} 
\le \abs{ \car_{\{k\}} \tensor e_j - \car \tensor e_j} 
=  \car_{\{k\}^c}. 
\]
This shows that 
\[ \inf_{g\in F_n} \abs{ \car_{\{k\}} \tensor e_j - g} 
= 0 \quad \text{on}\quad \{1, \dots, n\}.
\]
It follows that 
\[ \sup_{f\in M}  \inf_{g\in F_n} \abs{f - g} 
\le \sqrt{2} \, \car_{\{1, \dots, n\}^c} \searrow 0 \quad (n \to \infty),
\]
proving the (first) claim. 

For the proof of the second, observe that if $F = \{g_1, \dots, g_d\}
\subseteq E$ is a finite set and $n > d$, then the set
\[  \bigcup_{j=1}^d \Ball_H\bigl(g_j(n); \tfrac{1}{2}\sqrt{2}\bigr)
\]
cannot contain all the basis vectors $e_1, \dots, e_n$. 
Hence, there is $1\le i \le n$ with 
\[ \min_{j\le d} \norm{e_i - g_j(n)}_H \ge \tfrac{1}{2}\sqrt{2}.
\]
This leads to 
\[    \min_{j\le d} \abs{\car_{\{n\}} \tensor e_i - g_j} \ge
\tfrac{1}{2}\sqrt{2} \, \car_{\{n\}},
\]
violating uniform total order-boundedness of $M$. 
\end{example}

\smallskip

In the following lemma we collect some elementary 
properties of (uniformly) totally order-bounded sets.

\smallskip

\begin{lemma}\label{tob.l.properties}
Let $E, E_1, E_2$ be lattice-ordered spaces over a Stone algebra $\A$. 
\begin{aufzi} 
\item Each finite subset of $E$ is uniformly totally order bounded.
If $M \subseteq E$ is (uniformly) totally order-bounded, then so
  is each subset of $M$. 

\item If $M, N \subseteq E$ are (uniformly) totally order-bounded, then so are
  $M+N$ and $M \cup N$.

\item Suppose that $m \colon E \times E_1\to E_2$ is 
a bilinear mapping with $\abs{m(x,y)}\le \abs{x} \abs{y}$ for all
$x\in E, y\in E_1$. If $M\subseteq E$ and $N\subseteq E_1$ are
(uniformly) totally
order-bounded, then so is $m(M, N) \subseteq E_2$.

\item Let $M\subseteq E$ and suppose that $N \subseteq \A_+$ satisfies $\inf N = 0$
  ($\inf_{t\in N} 
  \norm{t}_{\A} = 0$) and for each
  $t\in N$ there is a (uniformly) totally order-bounded set $M_t\subseteq E$ with 
\[   M \subseteq M_t + \Ball_E[0; t].
\]
Then $M$ is (uniformly) totally order-bounded.

\item If $M \subseteq E$ is (totally) order-bounded, then so is 
its order-closure $\ocl(M)$.
\item Suppose
that  $T\colon E \to E_1$ is linear and order-bounded, i.e.
there is $c \in \R_{>0}$ with $\abs{Tx}\le c\abs{x}$ for all $x\in E$. Then
if $M\subseteq E$ is (uniformly) totally order-bounded, so is
$T(M)\subseteq E_1$.

\item If $M \subseteq E$ is bounded and $V \subseteq E$ is arbitrary,
  then  
\[     \inf \bigl\{ \sup_{x\in M} \inf_{y\in F} \abs{x-y} \bsuchthat F
\subseteq V\, \text{\rm finite}\bigr\}
= \inf \bigl\{ \sup_{x\in M} \inf_{y\in F} \abs{x-y} \bsuchthat F
\subseteq \ocl(V)\, \text{\rm finite}\bigr\}.
\]
In particular, if $M$ is totally order-bounded and $V$ is order-dense
in $E$, then 
\[    \inf_{F\in \Pot_\fin(V)} \sup_{x\in M} \inf_{y\in F} \abs{x-y} =0.
\]

\item Suppose that $E$ is a lattice-normed module, 
$r \in \R_{> 0}$ and $M \subseteq \Ball_E[0;r]$. 
Then for each finite set $F\subseteq E$ there is 
a finite set $F'\subseteq \Ball_E[0;2r]$ with
$\#F'\le \#F$ and 
\[     \inf_{y'\in F'} \abs{x-y} \le \inf_{y\in F} \abs{x-y}.
\]
In particular, one may replace
general finite subsets $F$ of $E$ by 
finite subsets of $\Ball_E[0;2r]$
in the definition of (uniformly)
totally order-bounded  sets. 
\end{aufzi} 
\end{lemma}

\begin{proof}
(a) is trivial. 

\prfnoi
(b)\ The proof of the statement about $M\cup N$ follows from the fact that 
in $\A$ from $u_\alpha  \searrow 0$ and $v_\alpha \searrow 0$ it
follows that $u_\alpha \vee v_\beta \searrow 0$. 
(This is obviously true if $u_\alpha, v_\alpha \in \R\car$, and this covers
the ``uniform'' case of the statement. But it is
even true in every vector lattice: if $h \le u_\alpha \vee v_\alpha$ for each
$\alpha$, then $h \le u_\beta \vee v_\alpha \le u_\beta + v_\alpha$,
and hence  $h - v_\alpha \le u_\beta$, for
$\beta \ge \alpha$. This yields $h-v_\alpha \le 0$ for all $\alpha$
and hence $h \le 0$. See also \cite[Thm.{} 15.9]{LuxZaa1}.)

Let $G, H\subseteq E$ be finite. Then 
\[  \abs{(a+b) - (z+w)} \le \abs{a-z} + \abs{b-w} \qquad (a\in M, b\in
N, z\in G, w\in H).
\]
This implies 
\[  \inf_{y\in F} \abs{(a+b) - y} \le \inf_{z\in G} \abs{a-z} 
+ \inf_{w\in H} \abs{b-w} \qquad (a\in M, b\in N)
\]
whenever $F \subseteq E$ is finite with $F \supseteq G + H$. And this
implies
\[  \sup_{x\in M{+}N} \inf_{y\in F} \abs{x-y} 
=  \sup_{a\in M, b\in N} \inf_{y\in F} \abs{(a+b) - y} \le \sup_{a\in
  M} \inf_{z\in G} \abs{a-z} 
+ \sup_{b\in N} \inf_{w\in H} \abs{b-w}
\]
for $G{+} H \subseteq F \subseteq E$ finite. This yields (b).

\prfnoi
(c)\ Define $A := \sup_{a\in M} \abs{a}$ and $B := \sup_{b\in N}
\abs{b}$.  Let $G \subseteq E$ and $H \subseteq E_1$ be finite. Then
\begin{align*}
   \abs{m(a,b) - m(z,w)} & = \abs{m(a,b-w) - m(a-z,b-w) + m(a-z, b)} 
\\ & \le A \abs{b-w} + \abs{a-z}\abs{b-w} + B\abs{a-z} 
\end{align*}
for $ a\in M,\: b\in N,\: z\in
     G,\: w\in H$. This implies
\begin{align*}
\inf_{y \in m(G,H)} \abs{m(a,b) - y} & =
\inf_{z\in G, w\in H} \abs{m(a,b) - m(z,w)}
\\ & \le 
A \inf_{w\in H} \abs{b-w} +  \inf_{z\in G} \abs{a-z}\, \inf_{w\in H} \abs{b-w} + B \inf_{z\in G}\abs{a-z}. 
\end{align*}
(Here we used that $\inf_\alpha f g_\alpha = f \inf_\alpha g_\alpha$ in $\A$
whenever $f\in \A_+$.)  From this we may pass to
\begin{align*}
  \sup_{x\in m(M,N)} & \inf_{y \in m(G,H)} \abs{x - y}
\\ & \le A \sup_{b\in N} \inf_{w\in H} \abs{b-w} +  \sup_{a\in M} \inf_{z\in G}
\abs{a-z}\: \sup_{b\in N} \inf_{w\in H} \abs{b-w} + B \sup_{A \in M} \inf_{z\in G}\abs{a-z}.
\end{align*}
And this yields (c).

\prfnoi
(d)\  Fix $t\in N$ and $F\subseteq E$ finite. For $x\in M$ let
$x_t\in M_t$ with $\abs{x- x_t} \le t$. Then 
\[  \abs{x- y} \le t + \abs{x_t- y} \qquad (y\in F)
\]
and hence 
\[ \inf_{y\in F} \abs{x- y} \le t + \sup_{z\in M_t} \inf_{y\in F}
\abs{z- y}.
\]
It follows that 
\[ \sup_{x\in M} \inf_{y\in F} \abs{x- y} \le t + \sup_{z\in M_t} \inf_{y\in F}
\abs{z- y},
\]  
and this implies (d).

\prfnoi
(e)\ Let $f \in \A_+$ and $F\subseteq E$ is finite. Then the set of
$x\in E$ satisfying
\[    \inf_{y\in F} \abs{x- y} \le f
\]
is order-closed in $E$ (simply because the mapping $x\mapsto
\inf_{y\in F} \abs{x-y}$ is order-continuous).  Hence, if each $x\in
M$ satisfies the inequality, then also each $x\in \ocl(M)$ does. This
proves the claim. 

\prfnoi
(f)\ This follows from 
\[  \inf_{y\in F}\abs{Tx - Ty}\le c \inf_{y\in F} \abs{x- y}
\qquad \text{for all $x\in M$ and all $F\subseteq E$ finite.}
\]
(g)\ Fix $d\in \N$. For $y\in V^d$, $z\in \ocl(V)^d$ and $x\in M$ one has
\[   \inf_{j=1,\dots, d} \abs{x- y_j}  \le \abs{x -z_k} + \sum_{j=1}^d
\abs{y_j - z_j} \qquad (k= 1, \dots, d).
\]
Taking first the infimum with respect to $k$ and then the supremum with respect
to $x\in M$ yields
\[   \sup_{x\in M} \inf_{j=1,\dots, d} \abs{x- y_j}
\le \sup_{x\in M}\inf_{j=1,\dots, d} \abs{x- z_j} +  \abs{z - y}_1
\]
with $\abs{\cdot}_1$ being the lattice-norm on $E^d$. By Lemma
\ref{app.l.ocl-prod}
(and an induction argument),  $\ocl(V)^d = \ocl(V^d)$, hence taking the
infimum over all $y\in V^d$ we obtain 
(see also the inclusion \eqref{app.eq.ocl})
\[   \inf_{y\in V^d} \sup_{x\in M} \inf_{j=1,\dots, d} \abs{x- y_j} 
\le \sup_{x\in M} \inf_{j=1,\dots, d} \abs{x- z_j}.
\]
And this implies
\[    \inf_{y\in V^d} \sup_{x\in M} \inf_{j=1,\dots, d} \abs{x- y_j} 
\le \inf_{z\in \ocl(V)^d} \sup_{x\in M} \inf_{j=1,\dots, d} \abs{x- z_j}.
\]
Finally, take the infimum with respect to $d\in \N$ to obtain the 
nontrivial inequality between the two quantities in the claim.
The additional statement follows readily. 

\prfnoi
(h)\ Let $F \subseteq E$ be an arbitrary finite set. Fix 
$y\in F$, let $p := \Bv{\abs{y} \le 2r\car} := \car_{\set{\abs{y}\le
    r}^\circ}$ and define $y' := py$.
Then $\abs{y'} \le  2r\car$ and 
$p^c2\abs{x} \le p^c 2r \le p^c\abs{y}$, and hence 
\[ \abs{x - y'} = p\abs{x-y} + p^c\abs{x}
\le p\abs{x-y} + p^c(\abs{y} - \abs{x}) 
\le \abs{x-y} \qquad (x\in M).
\]
It follows that for $x\in M$
\[ \inf_{y' \in F'} \abs{x-y'} \le \inf_{y\in F} \abs{x-y},
\]
where $F'= \{ y' \suchthat y\in F\}$. 
\end{proof}

For Kaplansky-Hilbert modules of finite rank (see \cite[Sec. 2.5]{EHK2024}) we obtain the following Heine--Borel-type theorem.

\smallskip

\begin{proposition}\label{tob.p.finrank}
Let $\A$ be a Stone algebra. Then $\Ball_\A[0;1]$ is 
uniformly totally order-bounded.  
More generally, let $E$ be a KH-module of finite rank over
$\mathbb{A}$. For $M \subseteq E$ the following assertions are equivalent:
		\begin{aufzii}
			 \item $M$ is order-bounded.\label{hb1}
			 \item $M$ is totally order-bounded.\label{hb2}
			 \item $M$ is uniformly totally order-bounded.\label{hb3}
		\end{aufzii}
\end{proposition}

\begin{proof}
Since $\A$ is a rank-one KH-module over itself, the first assertion is a
corollary of the second. However, we shall need the special
case in the proof of the more general statement. 

For $\veps \in \R_{> 0}$ find a finite set
$F\subseteq \C$ with $\Ball_{\C}[0;1] \subseteq \bigcup_{z\in F}
\Ball_{\C}[z;\veps]$. Then 
\[ \inf_{z\in F} \abs{f - z\car} \le \veps\car\qquad \text{for
each $f\in \Ball_\A[0; 1]$.}
\]
Now let $E$ be a KH-module over $\A$ and let
$e_1, \dots e_d$ be a suborthonormal basis for $E$ 
\cite[Sec.\,2.3]{EHK2024}. Since
$M\subseteq E$ is order-bounded, there is $c\in \R_{\ge 0}$ with
$\abs{f} \le c \car$ for all $f\in M$. Given $f\in M$ we may write
\[  f = {\sum}_{j=1}^d \lambda_j e_j 
\]
for certain $\lambda_j \in \abs{e_j} \A$. Then
\[  c^2\car \ge \abs{f}^2 = {\sum}_{j=1}^d \abs{\lambda_j}^2 
\]
and hence $\abs{\lambda_j} \le c\car$ for each $j=1, \dots, d$. This shows
that 
$M \subseteq \sum_{j=1}^d \Ball_\A[0; c] \cdot e_j$. 
Since $\Ball_\A[0;c]
\subseteq \A$ is uniformly  totally order bounded
(by what we already have shown), the claim follows 
from Lemma \ref{tob.l.properties}.
\end{proof}

In the measure-theoretic situation we have the following close connection between the two notions of total order-boundedness.

\begin{proposition}\label{tob.p.frkhm}
	Let $E$ be a KH-module over $\Ell{\infty}(\uY)$ for some
        probability space $\prY = (Y, \Sigma_\prY, \mu_\prY)$. Then
for a bounded subset $M \subseteq E$ the following assertions are equivalent:
		\begin{aufzii}
			\item $M$ is totally order-bounded.\label{normvsorder1}
			\item For every $\veps \in \R_{>0}$ there is a measurable set $A \in \Sigma_\prY$ with $\mu_\prY(A^c) \leq \varepsilon$ such that $\car_{A}M$ is uniformly totally order-bounded.\label{normvsorder2}
		\end{aufzii}
\end{proposition}

\begin{proof}
	Assume \ref{normvsorder1} and fix $\veps \in \R_{> 0}$. Note that 
in $\Ell{\infty}(\uY; \R)$ an infimum of a bounded subset must coincide
with the infimum of some countable subset (see, e.g. \cite[Thm. 7.6]{EFHN}). 
Hence, we find a sequence $(F_n)_{n \in \N}$ of finite subsets of $E$ such that
		\begin{align*}
                  0 = \inf_{F \in \Pot_{\fin}(E)} \sup_{x \in M} \inf_{y \in F} |x-y| = \inf_{n \in \N} \sup_{x \in M} \inf_{y \in F_n} |x-y|.
		\end{align*}
Without loss of generality we may suppose that $(F_n)_{n \in \N}$ is increasing, i.e., $F_n \subset F_{n+1}$ for all $n \in \N$. Then the sequence $(\sup_{x \in M} \inf_{y \in F_n} |x-y|)_{n \in \N}$ decreases and thus order-converges to zero. In particular, it converges to zero almost everywhere (see \cite[Lemma 7.5]{EHK2024}). By Egorov's theorem we thus find $A \in \Sigma_\uY$ with $\mu_\uY(A^c) \leq \veps$ such that the sequence
		\begin{align*}
  n \mapsto	\car_{A} \sup_{x \in M} \inf_{y \in F_n} |x-y|  = \sup_{x \in M} \inf_{y \in F_n} |\car_{A}x- \car_{A}y|
		\end{align*}
converges to zero in the norm of $\mathrm{L}^\infty(\uY)$.
	
Conversely,  assume \ref{normvsorder2} holds and set $c \coloneqq \sup \{\|x\|\mid x \in M\} \in [0,\infty)$. For every $n \in \N$ we then find $A_n \in \Sigma_\uY$ with $\mu_\uY(A_n^c) \leq \frac{1}{n}$ and a finite subset $F_n \subseteq E$ such that
		\begin{align*}
			 \sup_{x \in M} \inf_{y \in F_n} |\car_{A_n}x- y| \leq \tfrac{1}{n}\car.
		\end{align*}
	But then
		\begin{align*}
			\sup_{x \in M} \inf_{y \in F_n} |x - y| \leq \tfrac{1}{n}\car + c\car_{A_n^c}
		\end{align*}
	for every $n \in \N$. Since $\inf_{n \in \N} (\frac{1}{n}\car + c\car_{A_n^c}) = 0$, we obtain \ref{normvsorder1}.
\end{proof}

We conclude this section by comparing the two notions
of total order-boundedness with yet another similar property. For this
we suppose $E$ to be not just  a lattice-ordered space over $\A$, but a lattice-ordered {\em module}, see \cite[Def. 1.3]{EHK2024}.

\smallskip

\begin{definition}
Let $E$ be a lattice-ordered module over the commutative unital $C^*$-algebra
$\A$. For a finite set $F\subseteq E$ let
\[  \uZ_F := \Bigl\{ \sum_{y\in F} \lambda_y y \bsuchthat \lambda \in 
\A^F,\, \sup_{y\in F} \abs{\lambda_y} \le\car  \Bigr\}
\]
be the $\A$-{\emdf zonotope} in $E$ determined by $F$. 
A subset $M$ of $E$ has {\emdf property (CP)}
 if for each $\veps \in \R_{> 0}$ there is a finite set $F\subseteq E$ such that 
\[  M \subseteq \uZ_F + \Ball_E[0; \veps].
\]
\end{definition}

In the context of extensions of measure-preserving systems,
property (CP) was introduced  by 
Tao in \cite[Def. 2.13.7]{Tao2009}. Our terminology
is reminiscent of ``conditionally precompact'', which is
the term Tao uses. Cf. also Remark \ref{cpe.r.cpc} below.

\begin{proposition}\label{tob.p.zono}
Let $E$ be a lattice-normed module over a Stone algebra $\A$. 
Then a
subset $M \subseteq E$ has property \textup{(CP)} if and only if it is 
uniformly totally order-bounded. 
\end{proposition}

\begin{proof}
The first implication is rather trivial. If $M$ is uniformly totally
order-bounded and $\veps \in \R_{> 0}$, then there is a finite set $F\subseteq E$ with $\inf_{y\in F} \abs{x- y} \le \veps \car$ for each $x\in M$.
Hence, given $x\in M$ one finds idempotents $p_y = p_y^2 \in  \A$ with
$\sum_{y\in   F} p_Y = \car$ and   
$p_y \abs{x-y} \le \veps \car$ for each $y\in F$. 
It follows that $x \in \sum_{y\in F} p_y y + \Ball_E[0; \veps] \subseteq
\uZ_F + \Ball_E[0;\veps]$.

For the second implication note that 
\[ \uZ_F  = \sum_{y \in F} \Ball_\A[0;1] \cdot y 
\]
is uniformly totally order-bounded, by Lemma
\ref{tob.l.properties} and Proposition \ref{tob.p.finrank}. 
Hence, if $M$ has (CP) then $M$ is uniformly
totally order-bounded, again by Lemma \ref{tob.l.properties}.
\end{proof}

\section{Compact Extensions}\label{s.cpe}

In this section we apply the abstract terminology and results from
Section \ref{s.tob} to the structure theory of measure-preserving
dynamics. 
To wit, we characterize extensions with discrete spectrum as compact
extensions. In the classical case
of $\Z^d$-dynamics on separable probability spaces, this is due to
Furstenberg \cite[Theorem 6.13]{Furstenberg1981}. The general case is
due to Jamneshan \cite[Theorem 4.1]{Jam2023}.  

\medskip
We briefly describe the set-up, for details see
\cite{EHK2024}. Let $G$ be a group (any group, no topology).  A {\emdf
  measure-preserving} 
$G$-system is pair $(\prX; T)$ consisting of a probability space
$\prX$ and a representation $T= (T_t)_{t\in G} \colon G \to
\BL(\Ell{1}(\prX))$
of $G$ as {\em Markov embeddings} (= unital, integral-preserving
linear lattice homomorphisms).  An {\emdf extension} of 
two measure-preserving systems $(\prY;S), (\prX;T)$ is a Markov
embedding
$J\colon \Ell{1}(\prY) \to \Ell{1}(\prX)$ with $JS_t = T_tJ$ for all $t\in
G$. We write $J\colon (\prY;S)\to (\prX;T)$ to denote
the extension, or just $\prX|\prY$, for the sake of simplicity.

Given an extension $\prX|\prY$, the space $\Ell{2}(\prX)$ becomes an
$\Ell{\infty}(\prY)$-module in a natural way. (Basically, $J$
identifies
$\Ell{2}(\prY)$ with 
$\Ell{2}(X, \calF,\mu_\prX)$ for some sub-$\sigma$-algebra $\calF$ of $\Sigma_\prX$.)
Along with the  extension comes the {\emdf conditional expectation operator} 
$\Exp_\prY\colon \Ell{1}(\prX) \to \Ell{1}(\prY)$, which on the level of
$\Ell{2}$-spaces is nothing but the adjoint
$\Exp_{\prY} = J^*$ of $J$.  It coincides with the classical conditional expectation under
the identification of $\Ell{1}(\prY)$ with its range under $J$. 

It turns out that the space
\[   \Ell{2}(\prX|\prY) := \{ f\in \Ell{2}(\prX) \suchthat 
\Exp_{\prY} \abs{f}^2 \in \Ell{\infty}(\prY) \}
\]
is a Kaplansky--Hilbert module over the Stone algebra $\A :=
\Ell{\infty}(\prY)$ with respect to the inner product
\[  \sprod{f}{g}_\prY := \Exp_\prY (f\konj{g}),
\]
see \cite[Prop.\,7.6]{EHK2024}.
The lattice-valued norm of an element $f\in \Ell{2}(\prX|\prY)$ is
\[ \abs{f}_\prY := \sqrt{\sprod{f}{f}_\prY} = \sqrt{ \Exp_\prY \abs{f}^2}.
\]
Note the difference between $\abs{f}_\prY$ and $\abs{f}$, the latter
being the usual (pointwise) modulus of $f$. 

Since $\Ell{2}(\prX|\prY)$ is $G$-invariant and $\sprod{T_tf}{T_t g}_\prY = S_t
\sprod{f}{g}_\prY$, one arrives at a {\em KH-dynamical system} as
defined in \cite[Sec.\,5.1]{EHK2024}, see also \cite[Sec.\,7.2]{EHK2024}.

\medskip
On $\Ell{2}(\prX|\prY)$ we have three 
natural notions of convergence: convergence with respect to the norm
\[  \norm{f}_{\Ell{2}(\prX|\prY)} = \norm{ \abs{f}_{\prY}
}_{\Ell{\infty}(\prY)},
\]
convergence with respect to the $\Ell{2}$-topology as a subspace of
$\Ell{2}(\prX)$, and order-convergence (see Appendix \ref{s.app}).  
The following result 
relates these notions.

\begin{lemma}\label{cpe.l.norm-order}
Let $\prX|\prY$ be an extension of measure--preserving systems.
Then the following assertions hold:
\begin{aufzi}
\item Within $\Ell{2}(\prX|\prY)$ each norm convergent net
and each order-convergent net is $\Ell{2}$-convergent
(to the same limit).  
\item If $(f_n)_n$ is a sequence in $\Ell{2}(\prX|\prY)$
which $\Ell{2}$-converges to $f\in \Ell{2}(\prX|\prY)$, then there is
a subsequence $(h_n)_n$ with the following property: 
for each $\delta \in \R_{> 0}$ there is $E\in \Sigma_\prY$ with
$\mu_\prY(E^c) < \delta$ and $\car_E h_n \to \car_E f$ in the
norm of $\Ell{2}(\prX|\prY)$.

\item If $M \subseteq \Ell{2}(\prX|\prY)$ is a
  $\Ell{\infty}(\prY)$-submodule, then in $\Ell{2}(\prX|\prY)$ its order-closure 
coincides with its $\Ell{2}$-closure.
In particular, 
$M$ is order-dense in $\Ell{2}(\prX|\prY)$ iff 
$M$ is dense in $\Ell{2}(\prX)$. 
\item If $M \subseteq \Ell{2}(\prX|\prY)$ is (uniformly) totally
  order-bounded, then so are the sets 
\[ \abs{M} := \{ \,\abs{f}\,  \suchthat  f\in M\} \quad
\text{and}\quad M' := \{ \konj{f} \suchthat f\in M\}. 
\] 
\end{aufzi}
\end{lemma}

\begin{proof}
For (a) see \cite[Lemma 7.5(ii)]{EHK2024} and its proof, for (b) see
the proof of \cite[Lemma 7.5(iii)]{EHK2024}. (c) is a consequence of (b)
as in \cite[Lemma 7.5.(iii)]{EHK2024} and the fact that
$\Ell{2}(\prX|\prY)$ is dense in $\Ell{2}(\prX)$. For the proof of (d)
we observe that for $f,g\in \Ell{2}(\prX|\prY)$ one has 
\[  \abs{\abs{f} - \abs{g}}_\prY \le \abs{f-g}_\prY
\]
as a consequence of the reverse triangle inequality for the modulus.
And this implies 
\[  \sup_{u\in \abs{M}} \inf_{v\in \abs{F}} \abs{u-v}_\prY
= \sup_{f\in M} \inf_{g\in F} \abs{\abs{f} - \abs{g}}_\prY
\le \sup_{f\in M} \inf_{g\in F} \abs{f-g}_\prY
\]
for any finite $F \subseteq \Ell{2}(\prX|\prY)$. The proof
for $M'$ is similar. 
\end{proof}

\medskip
We proceed with recalling a definition from 
\cite{EHK2024}\footnote{Actually, the definition in \cite{EHK2024} 
is slightly different, but proved to be equivalent in \cite[Prop. 8.5]{EHK2024}.}.

\begin{definition}
For an extension $\prX|\prY$ of measure-preserving $G$-systems
its {\emdf (relative) Kronecker subspace}
is 
\[ \scrE(\prX|\prY) :=  \cl_{\Ell{2}} \bigcup
\bigl\{ \text{finitely generated $G$-invariant
  $\Ell{\infty}(\prY)$-submodules
of $\Ell{2}(\prX)$}\bigr\}.
\]
The extension $\prX|\prY$  has  {\emdf discrete spectrum} if 
$\scrE(\prX|\prY) = \Ell{2}(\prX)$.
\end{definition}

In \cite{EHK2024}, extensions with discrete spectrum were approached through
the theory of KH-dynamical systems. In particular, it was proved there
that 
\[  \scrE(\prX|\prY) = \cl_{\Ell{2}}\,  \FM_T(\prX|\prY)
\]
where 
\[ \FM_T(\prX|\prY) := 
\bigcup \bigl\{ \text{$G$-invariant finite-rank KH-submodules of $\Ell{2}(\prX|\prY)$}\bigr\}, 
\]
see \cite[Prop. 8.5]{EHK2024}.

\vanish{

Theorem 6.9 from
\cite{EHK2024} yields the $\prY$-orthogonal decomposition
\[ \Ell{2}(\prX|\prY) = 
\Ell{2}(\prX|\prY)_{\ds} \oplus \Ell{2}(\prX|\prY)_{\wm} 
\]  
where $\Ell{2}(\prX|\prY)_{\ds}$ is the
``discrete spectrum part'' (generated
by the invariant finite-rank KH-submodules) of the associated KH-system, and
$\Ell{2}(\prX|\prY)_{\wm}$ is the ``weakly mixing part'',
cf. also \cite[Sec.s\,8.2 and 8.3]{EHK2024}.  One has
\[ \scrE(\prX|\prY) = \cl_{\Ell{2}} \Ell{2}(\prX|\prY)_{\ds}
\]
and hence the extension $\prX|\prY$ has discrete spectrum iff 
the weakly mixing part is $\{0\}$, see  \cite[Prop. 8.5]{EHK2024}.  
}

\medskip
Next, we define compact extensions.

\begin{definition}
Let $\prX|\prY$ be an extension of measure preserving
$G$-systems $(\prX;T)$ and $(\prY;S)$.  
An element $f\in \Ell{2}(\prX|\prY)$
is said to be {\emdf conditionally almost periodic} if its {\emdf orbit} $T_Gf :=
\{ T_t f \suchthat t\in G\}$ is uniformly
totally order-bounded in $\Ell{2}(\prX|\prY)$, i.e., if for each
$\veps \in \R_{>0}$ there is a finite set $F\subseteq
\Ell{2}(\prX|\prY)$ such that 
\[    \inf_{g\in F} \abs{T_t f - g}_\prY \le \veps \car \quad
\text{for all $t \in G$.}
\]
The extension $\prX|\prY$ is {\emdf compact}, if the
set 
\[ \AP_T(\prX|\prY) := \{ f\in \Ell{2}(\prX|\prY) \suchthat 
\text{$f$ is conditionally almost periodic}\}
\]
is dense in $\Ell{2}(\prX)$.
\end{definition}

\begin{remark}\label{cpe.r.cpc}
If $\prX|\prY$ is an extension then it is common
to call a subset $M \subseteq \Ell{2}(\prX|\prY)$
{\emdf conditionally precompact} if it
is uniformly totally order-bounded. By Proposition
\ref{tob.p.zono}, this is consistent with 
Tao's definition of this term in \cite[Def.\,2.13.7]{Tao2009}.
Hence, with this
terminology a function $f$ is conditionally almost
periodic if its orbit is conditionally 
precompact. (And this, still,  is consistent with 
 \cite[Def.\,2.13.7]{Tao2009}.) 
\end{remark}

The following is a straightforward
consequence of Lemma \ref{tob.l.properties}(d)  and Lemma \ref{cpe.l.norm-order}.

\begin{proposition}\label{cpe.p.ap-module}
Given an extension $\prX|\prY$ of measure-preserving
$G$-systems, the sets
\[ \AP_T(\prX|\prY)\quad \text{and}\quad 
\{ f\in \Ell{2}(\prX|\prY) \suchthat 
T_Gf \,\,\text{\rm is totally order-bounded}\}
\]
are norm-closed $G$-invariant, conjugation- and modulus-invariant
 $\Ell{\infty}(\prY)$-submodules  of $\Ell{2}(\prX|\prY)$. 
\end{proposition}

With these preparations at hand, we are now in the position
to formulate and prove the main result of this section.

\begin{theorem}\label{cpe.t.cpds}
Let $\prX|\prY$ be an extension of measure-preserving systems. Then
the following assertions hold:
\begin{aufzi}
\item $\dps
\FM_T(\prX|\prY) \subseteq \AP_T(\prX|\prY) \subseteq \{ f\in
\Ell{2}(\prX|\prY) \suchthat T_G f \,\,\text{totally order-bounded}\}$.
\item  The $\Ell{2}$-closures of all three sets in {\rm (a)} coincide
  with $\scrE(\prX|\prY)$.
\item  For each $f\in \Ell{2}(\prX)$:
 \[  f\in \scrE(\prX|\prY) \quad\iff\quad 
\forall\, \delta \in \R_{>0}\,\,
 \exists\, E \in \Sigma_\uY : \,\, \mu_\uY(E^c) \leq
  \delta\,\, \wedge\,\,  \car_E f \in \AP_T(\prX|\prY).
\]
\end{aufzi}
\end{theorem}

\begin{proof}
(a)\, If $M$ is a $G$-invariant, finite-rank KH-submodule of $\Ell{2}(\prX|\prY)$
and $f\in M$, then the orbit $T_Gf$ is a bounded subset of $M$ and
hence uniformly totally order-bounded, by Proposition \ref{tob.p.finrank}. Hence
$\FM_T(\prX|\prY) \subseteq \AP_T(\prX|\prY)$.
The second inclusion is trivial, as each uniformly totally
order-bounded set is totally order-bounded.

\prfnoi
(c)\, We shall prove the equivalence
\[ 
  f\in \cl_{\Ell{2}}\AP_T(\prX|\prY) \,\iff\, 
\forall\, \delta \in \R_{>0}\,\,
 \exists\, E \in \Sigma_\uY : \,\, \mu_\uY(E^c) \leq
  \delta\,\, \wedge\,\,  \car_E f \in \AP_T(\prX|\prY).
\]
Then (c) follows as soons as we have proved (b) (see below). 

The implication ``$\Leftarrow$'' is straightforward. 
For the converse, let  $f\in \cl_{\Ell{2}}\AP_T(\prX|\prY)$ and
$\delta \in \R_{>0}$. In the first step 
we suppose in addition that $f\in \Ell{2}(\prX|\prY)$. 
By assumption  we find a sequence
$(f_n)_n$ in $\AP_T(\prX|\prY)$ with $f_n \to f$ in $\Ell{2}$.  Since 
$f\in \Ell{2}(\prX|\prY)$, by
Lemma \ref{cpe.l.norm-order} 
we may pass to a subsequence
$(h_n)_n$ such that there
is a subset $E\in \Sigma_\prY$ with $\mu_\prY(E^c) < \delta$ 
and $\car_E h_n \to \car_Ef$ in the norm of
$\Ell{2}(\prX|\prY)$. Since $\AP_T(\prX|\prY)$ is a norm-closed
submodule of $\Ell{2}(\prX|\prY)$, it follows that 
$\car_E f\in \AP_T(\prX|\prY)$ as required.  

Finally, for general $f\in \cl_{\Ell{2}}\AP_T(\prX|\prY)$ 
we can find $B\in \Sigma_\prY$ such that $\car_B f\in \Ell{2}(\prX|\prY)$
and $\mu_\prY(B^c) < \frac{\delta}{2}$. 
(For example, let $B \coloneqq \set{ \Exp_\prY\abs{f}^2 <
  N}$ for $N> 0$ large.) Since $\AP_T(\prX|\prY)$ is a 
$\Ell{\infty}(\prY)$-module, $\car_B f \in
\cl_{\Ell{2}}\AP_T(\prX|\prY)$, and we can apply what we have already
proved (with $\delta$ replaced by $\frac{\delta}{2}$). This yields a set
$C \in \Sigma_\prY$ with $\mu_\prY(C)< \frac{\delta}{2}$ such that 
$\car_{B\cap C}f = \car_C(\car_B f) \in \AP_T(\prX|\prY)$. Taking
$E := B \cap C$ then yields what is desired.

\prfnoi
(b)\ By the inclusions in (a) it suffices to establish the implication
\[ f\in \Ell{2}(\prX|\prY),\, T_G f \,\,\text{totally order-bounded}
\quad\dann\quad f\in \scrE(\prX|\prY).
\]
By Theorem 6.9 in  \cite{EHK2024} we may write
\[\Ell{2}(\prX|\prY) = 
\Ell{2}(\prX|\prY)_{\ds} \oplus \Ell{2}(\prX|\prY)_{\wm} 
\] 
as an orthogonal sum of KH-modules, 
where $\Ell{2}(\prX|\prY)_{\ds} = \ocl\bigl(\FM_T(\prX|\prY)\bigr)
\subseteq \scrE(\prX|\prY)$. 

Let $Q$ denote the orthogonal projection onto the 
KH-submodule $\Ell{2}(\prX|\prY)_{\wm}$. Then $QT_t = T_tQ$ for each
$t\in G$, and hence $T_GQf = QT_G f$ is totally order-bounded
whenever $T_G f$ is, 
by Lemma \ref{tob.l.properties}(f). Therefore,  we are reduced to show
\[   f\in \Ell{2}(\prX|\prY)_{\wm},\, T_G f \,\,\text{totally order-bounded}
\quad\dann\quad f =0.
\]  
In order to do this, pick $f$ satisfying the stated hypotheses.
In addition, we may and do suppose that $\abs{f}_\prY \le 1$.

Fix $\veps \in \R_{> 0}$. Since $\Ell{\infty}(\prX)$
is order-dense in $\Ell{2}(\prX|\prY)$, by
Lemma \ref{tob.l.properties}(g)  
the net 
\[  F \mapsto \sup_{t\in G} \inf_{g\in F} \abs{T_t f - g}_\prY 
\]
decreases to $0$, where $F$ ranges over the finite subsets of
$\Ell{\infty}(\prX)$. Since the $\Ell{1}(\prY)$-norm 
is order-continuous \cite[Sec.\,7.2]{EFHN}, there is a finite set 
$F \subseteq \Ell{\infty}(\prX)$ with 
\[  \int_\prY \sup_{t\in G} \inf_{g\in F} \abs{T_t f - g}_\prY  \le
\veps.
\]
Now for $t\in G$ and $h\in F$ we certainly have
\[ \abs{T_tf}_\prY^2
\le  \abs{ \sprod{T_t f}{T_tf - h}_\prY } 
+ \abs{ \sprod{T_t f}{h}_\prY } 
\le  \abs{T_tf - h}_\prY + \sum_{g\in F} \abs{\sprod{T_t f}{g}_\prY},
\]
and taking the infimum over $h$ we arrive at
\[ \abs{T_tf}_\prY^2 
\le \inf_{g\in F} \abs{T_t f - g}_\prY + 
\sum_{g\in F} \abs{\sprod{T_t f}{g}_\prY}
\le \sup_{s\in G} \inf_{g\in F} \abs{T_s f - g}_\prY
+ 
\sum_{g\in F} \abs{\sprod{T_t f}{g}_\prY}.
\] 
Integrating over $\prY$ we find, with $n := \#F$ being the
cardinality of $F$, 
\[ \norm{f}_2^2 = \norm{T_t f}_2^2 
= \int_\prY \abs{T_t f}_\prY^2 \le \veps + 
\sum_{g\in F} \norm{ \sprod{T_t f}{g}_\prY }_1 
\le \veps + n^\frac{1}{2} 
\bigl(\sum_{g\in F} \norm{ \sprod{T_t f}{g}_\prY }_2^2 \bigr)^\frac{1}{2} 
\]
for each $t\in G$. Hence, it suffices to prove
\[  \inf_{t\in G} \sum_{g\in F} \norm{ \sprod{T_t f}{g}_\prY }_2^2 = 0.
\]
But this is true by \cite[Prop.\,8.8]{EHK2024}, since  $F\subseteq \Ell{\infty}(\prX)$.
\end{proof}

\medskip
As a corollary we obtain the announced characterization
of extensions with discrete spectrum.

\begin{corollary}\label{cpe.c.cpds}
  Let $\prX|\prY$ be an extension of measure-preserving systems. Then the following assertions are equivalent.
\begin{aufzii}
\item The extension $\prX|\prY$  has discrete spectrum.\label{compactisometric1}
\item The extension $\prX|\prY$ is compact,
                          i.e., the space $\AP_T(\prX|\prY)$ is dense
                          in $\Ell{2}(\prX)$.\label{compactisometric2}

\item The set $\{ f\in \Ell{2}(\prX|\prY) \suchthat 
\textrm{$T_Gf$ is totally order-bounded}\}$ is dense
in $\Ell{2}(\prX)$.

\item For every $f \in \Ell{2}(\uX)$ and every $\delta \in \R_{>0}$
  there is a measurable set $E \in \Sigma_\uY$ with $\mu_\uY(E^c) \leq
  \delta$ such that $\car_E f \in \AP_T(\prX|\prY)$.
\label{compactisometric5}
		\end{aufzii}
\end{corollary}

\begin{remark}\label{cpe.r.cpe-KH-char?}
By Lemma \ref{cpe.l.norm-order}
 we can reformulate assertion (ii)
in Corollary \ref{cpe.c.cpds} as 
\begin{enumerate}
\item[(ii)'] {\em The set $\AP_T(\prX|\prY)$ is order-dense
in $\Ell{2}(\prX|\prY)$.}
\end{enumerate}
And (iii) can be reformulated as
\begin{enumerate}
\item[(iii)'] {\em The set $\{ f\in \Ell{2}(
\prX|\prY) \suchthat \text{$T_Gf$ is totally order-bounded}\}$
is order-dense in $\Ell{2}(\prX|\prY)$.}
\end{enumerate}
Since also (i) can be reformulated entirely in terms
of KH-modules, the equivalence of (i), (ii), and (iii) could
actually be formulated for arbitrary KH-dynamical systems. 
However, whereas the implications (i)$\dann$(ii)$\dann$(iii)
are still valid in this more general situation, 
we do not know whether this is also true for the
implication (ii)$\dann$(i). Our proof 
of Corollary \ref{cpe.c.cpds} hinges on the
Birkhoff--Alaoglu theorem, which refers to the global
$\Ell{2}$-dynamics, a feature not present in the
abstract framework of KH-dynamical systems.
\end{remark}

\section{Relative Cyclical Compactness}\label{s.rcc}

In this  section we shall
relate the notions of total order-boundedness and
{\em relative cyclical compactness} as defined in \cite{Kusr2000}. 
 For this we boldly assume the reader to be familiar with the relevant notions.

\bigskip

As before, let  $E$  be a lattice-normed space
over the Stone algebra $\A$. We let $\B$ be the
complete Boolean algebra of idempotents on $\A$ and freely
identify $\A$ with $\Ce(\Om)$ and $\B$ with $\Clop(\Om)$, where $\Om$ is the
(extremally disconnected) Gelfand space of $\A$. 

There is a canonical $\B$-set structure\footnote{For definition of a $\B$-set see \cite[{}A.12]{Kusr2000}. Note, however, 
that Kusraev uses the symbol ``$\Bv{x=y}$'' for something
different than we do.}  on $E$ defined by 
\[ \Bv{x\neq y} := \supp(\abs{x-y}),\qquad \Bv{x=y} := 
\Bv{x\neq y}^c\qquad  (x,y \in E).
\]
Here we employ the notion of the {\em support} of
an element of $\A$ as defined in  \cite[Section 2.1]{EHK2024}. It was mentioned there that under the identification $\A = \Ce(\Om)$
one has $\Bv{x=y} =  \car - \supp(\abs{x-y}) = \car_{\set{ \abs{x -y} = 0}^\circ }$. 

In this terminology, Kusraev's axioms for a $\B$-set read $(x,y,z\in E)$: 
\[  \Bv{x=y} = \Bv{y=x}, \quad \Bv{x=y}= \car \,\,\gdw\,\,
x=y,\quad \Bv{x=y}\Bv{y=z} \le \Bv{x=z}.
\]
Observe that one can take $E = \A$ here. 

Suppose $F,G$ are $\B$-sets, then a mapping 
$f: F\to G$ is a {\emdf $\B$-set map} (a ``non-expanding'' map in
Kusraev's terminology) if
\[  \Bv{x = y} \le \Bv{f(x)=f(y)} \qquad (x,y\in F).
\]
It is easy to see that for $z\in E$ the  mapping 
\[   E \to \A,\qquad x\mapsto \abs{x-z},
\]
is a  $\B$-set map.

A family  $(p_\alpha)_\alpha$ in $\B$  is a {\emdf partition of unity}
if 
\[  p_\alpha \wedge p_\beta = 0 \quad \text{whenever $\alpha \neq
  \beta$}\quad \text{and}\quad \bigvee_\alpha p_\alpha = \car.
\]
Note that we allow $p_\alpha =0$ in this definition. We recall the following (well-known) result.

\begin{lemma}[Exhaustion Principle/Disjointification]\label{rcc.l.disjoint}
Let $(t_\alpha)_\alpha$  be any family in $\B$ with $\bigvee_\alpha
t_\alpha = \car$. Then there is a partition of unity
$(p_\alpha)_\alpha$
with $p_\alpha \le t_\alpha$ for each $\alpha$.  
\end{lemma}

\begin{proof}
Let $E := \{ e\in \B \suchthat \exists\,\alpha: e \le
t_\alpha\}$ and let $S\subseteq E$ be a maximal antichain, i.e. a
maximal subset of pairwise disjoint elements. (This exists by a
standard application of Zorn's lemma.) By maximality, each upper bound
of $S$ is also an upper bound of all the $t_\alpha$, hence
$\bigvee S = \car$.   

For each $s\in S$ let
$\alpha(s)$ be an index with  $s\le t_{\alpha(s)}$. (This is a
direct application of the Axiom of Choice.) Now define
\[   p_\alpha := \bigvee\{ s \suchthat \alpha(s) = \alpha\}
\]
for each index $\alpha$ (with $\bigvee \leer = 0$ as usual). Then it is routine to verify that $(p_\alpha)_\alpha$ has the required properties.

\prfnoi
Alternative proof (sketch):
By the well-ordering principle we may suppose that the index set is well-ordered. Define
\[ q_\alpha :=  \bigvee \{ t_\beta \suchthat
\beta < \alpha \} \quad \text{and}\quad 
  p_\alpha := t_\alpha \wedge q_\alpha^c
\]
for each $\alpha$. 
Then, clearly, the $p_\alpha$ are pairwise disjoint. 
Using the defining property of a well-ordering one shows that 
\[  t_\alpha \le \bigvee_{\beta \le \alpha} p_\alpha
\]
is true for all $\alpha$, and hence $\bigvee_\alpha p_\alpha = \car$
as desired.


\prfnoi
Third proof (sketch): Let 
\[ \calM := 
  \{(s_\alpha)| \forall \alpha: \B \ni s_{\alpha} \leq t_{\alpha}\,\, \text{and}\,\, \forall \alpha, \beta: \alpha \neq \beta \Rightarrow s_{\alpha} \wedge s_{\beta} = 0\}.
\]
The set $\calM$ is partially ordered (componentwise). If 
$(p_\alpha)_\alpha$ is a maximal element of $\calM$ (which exists by
Zorn's lemma) then it satisfies the requirements.
\end{proof}

\begin{remark}
The first step in the (first) proof of Lemma \ref{rcc.l.disjoint} is 
Corollary (1) of
Kusraev's {\em exhaustion principle}, see
\cite[{}1.1.6]{Kusr2000}. 
Although not stated explicitly, the lemma is applied a couple of times
in \cite{Kusr2000}
(e.g. in the proof of Thm.{ }5.3.6). And the remaining
arguments are given in the proof of Thm.{ }8.1.8. 

The second proof is the transfinite version of the usual
``disjointification'' procedure known from elementary measure theory
courses. One can find it, basically, in Sikorski's book \cite[Thm.{ }2.20.2]{SikorskiBA}.   
\end{remark}

Let $(p_\alpha)_\alpha$ be a partition of unity
 in $\B$ and $(x_\alpha)_\alpha$
a  family (over the same index set) of elements in $E$. An element $x\in E$ is called the {\emdf
  mixing} of $(x_\alpha)_\alpha$ over $(p_\alpha)_\alpha$ if 
\[    p_\alpha \le \Bv{x = x_\alpha} \quad \text{for all $\alpha$}.
\]
The mixing element $x$ is uniquely determined by this condition and one writes
$x = \sum_\alpha p_\alpha x_\alpha$. 

Mixings are preserved under $\B$-set maps. In the particular case
of the $\B$-set map $x\mapsto \abs{x-z}$ from $E$ to $\A$, this means
\[ x = \sum_\alpha p_\alpha x_\alpha \quad \dann\quad  
\abs{z- x} = \sum_{\alpha} p_\alpha \abs{x_\alpha - z}.
\]
Of course, a mixing need not exist. 
A subset $M$ of $E$ is called {\emdf mix-complete} if every mixing of
elements of $M$
(i.e., each mixing of any family in $M$ over any partition of unity of
$\B$) exists in $M$. And it is called
{\emdf $\B$-cyclic} (by Kusraev \cite[{}7.3.3]{Kusr2000}) or {\emdf boundedly mix-complete} (by us) 
if every mixing of any {\em bounded} family in $M$ exists in $M$.
If mixings exist in $M$ just for all {\em finite} partitions of unity, 
the set $M$ is called {\emdf finitely mix-complete}.

Note that, in the expression $x= \sum_\alpha p_\alpha x_\alpha$, the
sum 
is to be understood in a  purely formal way,
just expressing that $x$ is the mixing of the bounded family
$(x_\alpha)_\alpha$ over the partition of unity
$(p_\alpha)_\alpha$. However, if $E$ is a lattice-normed {\em module}
one may interpret the sum as the  net of partial sums
\[  F \mapsto \sum_{\alpha \in F} p_\alpha x_\alpha \qquad (F\subseteq
\Lambda \,\text{finite}),
\]
and the identity $x= \sum_\alpha p_\alpha x_\alpha$ 
as telling that this net is
order-convergent to $x$. In particular, each lattice-normed module is
finitely mix-complete, and each Kaplansky--Banach module is boundedly mix-complete
(=$\B$-cyclic).

The following characterization links all these notions.

\begin{theorem}\label{rcc.t.KBM-char}
For a lattice-normed
space $E$ over a Stone algebra $\A$ the following assertions are equivalent:
\begin{aufzii}
\item $E$ is finitely mix-complete and order-complete;
\item $E$ is boundedly mix-complete (=$\B$-cyclic) and norm-complete;
\item $E$ is a Kaplansky-Banach module over $\A$ (for some (unique)
  multiplication $\A \times E \to E$);

\item $E$ is a Banach--Kantorovich space (=decomposable, order-complete, lattice-normed space {\rm \cite[2.2.1]{Kusr2000}});
\item $E$ is disjointly decomposable and order-complete.
\end{aufzii}
\end{theorem}

\begin{proof}
The equivalence of (i) and (ii) is basically \cite[Thm. 2.2.3]{Kusr2000}; the implication (i),(ii)$\dann$(iii) follows from \cite[{}2.1.8]{Kusr2000}; (iii)$\dann$(iv) holds since
each Kaplansky--Banach module is decomposable (in the sense of \cite[{}2.1.1]{Kusr2000}); (iv)$\dann$(v) is trivial. 

We prove the implication (v)$\dann$(i). Suppose that 
$p \in \B$ and $x_1, x_2 \in E$. For each $j=1,2$ we can write
$\abs{x_j} = p\abs{x_j} + p^c\abs{x_j}$. 
Disjoint decomposability
yields elements $u_j, v_j\in E$ with
$x_j = u_j + v_j$ and $\abs{u_j} = p\abs{x_j}$, $\abs{v_j}= p^c\abs{x_j}$. Define $x := u_1 + v_2$. Then it follows
\[  \Bv{x = x_1} \ge p\quad \text{and}\quad \Bv{x= x_2} \ge p^c
\]
and hence $x = px_1 + p^cx_2$ in the sense of mixings.
\end{proof} 

Let again $E$ be a lattice-normed space  over $\A$.
The {\emdf mix-closure} of $M\subseteq E$ in $E$  is 
\[ \mix(M) = \Bigl\{ x\in E \suchthat \exists 
(p_m)_{m\in M} \,\, \text{partition of unity} : x = \text{${\sum}_{m\in M}$} p_m m\Bigr\}.
\]
And $M$  is {\emdf  mix-closed}  in $E$ if $\mix(M) = M$. It is straightforward
to prove that $\mix(\mix(M)) = \mix(M)$, and hence $\mix(M)$ is
mix-closed whatever $M$ is. 

\medskip
We are now approaching the main result of this section. Formulating it
requires the notion of a {\emdf (relatively) cyclically compact} subset, see
\cite[{}8.5.1]{Kusr2000}. However, we shall not work 
with the original definition, but with the following equivalent characterization
provided by \cite[Theorem 8.5.2]{Kusr2000}.

\begin{lemma}\label{rcc.l.rcc-char}
Let $E$ be a Kaplansky--Banach module. A mix-complete subset  $M \subseteq E$ 
is relatively cyclically compact in $E$ iff it has the following property:

\smallskip
For each $\veps \in \R_{> 0}$ there is a countable
partition of unity $(q_n)_n$ in $\B$ and a sequence 
$(F_n)_n$ of finite subsets of $E$ such that 
for each $x\in M$ there is $z_n\in \mix(F_n)$  
with $q_n \abs{x- z_n} \le \veps\car$.
\end{lemma}

\begin{proof}
See \cite[Theorem 8.5.2]{Kusr2000} and its proof. The original formulation in \cite[{}8.5.2]{Kusr2000} requires the finite sets $F_n$ to be subsets of $M$ itself. However, the proof literally
yields only $F_n \subseteq E$ and not $F_n \subseteq M$.
\end{proof}

\begin{proposition}\label{rcc.p.tob-rcc}
Let $E$ be a Kaplansky--Banach module over a Stone algebra $\A$. Then
for $M \subseteq E$ the following assertions are equivalent.
\begin{aufzii}
\item $M$ is totally order-bounded.
\item $M$ is bounded and $\mix(M)$ is relatively cyclically compact.
\end{aufzii}
\end{proposition}

\begin{proof}
(i)$\dann$(ii): Since $M$ is totally order-bounded, it is bounded.  
For each $F \subseteq E$ the mapping $x\mapsto \inf_{y \in F} \abs{x-y}$ 
is  a $\B$-set map and hence
respects mixings. Therefore
\[  \sup_{x\in M} \inf_{y\in F} \abs{x-y} = \sup_{z\in \mix(M)}
\inf_{y\in F}\abs{z-y}  
\]
for each finite $F\subseteq E$. It follows that $\mix(M)$ is totally
order-bounded.  In particular, since $E$ is boundedly mix-complete and
$\mix(M)$ is mix-closed in $E$, $\mix(M)$ is mix-complete. 
So we may suppose without loss of generality that $M$ is mix-complete.

In order to show that $M$ is relatively cyclically compact, we 
apply Lemma \ref{rcc.l.rcc-char}. By hypothesis,
\[  \inf_{F\in \Pot_\fin(E)} \sup_{x\in M} \inf_{y \in F} \abs {x- y} = 0.
\]
As $M$ is bounded we may apply Lemma \ref{tob.l.properties}(h) 
and find $r \in \R_{> 0}$ such that 
\[  \inf_{F\in \Pot_\fin(E_r)} \sup_{x\in M} \inf_{y \in F} \abs {x- y} = 0,
\]
where $E_r \coloneqq \Ball_E[0; 2r]$, for short.

Now fix $\veps \in \R_{> 0}$. For each $F\in \Pot_\fin(E_r)$ let 
$h_F := \sup_{x\in M} \inf_{y \in F} \abs {x- y} \in \A_+$. 
Define $\tilde{p}_F := \set{ h_F \le \veps}^\circ \in\B$. Then $h_F
\ge \veps$ on $\tilde{p}_F^c$, and hence $\bigvee_F \tilde{p}_F = \car$.
By disjointification (Lemma \ref{rcc.l.disjoint}) 
there is a partition of unity $(p_F)_{F\in
  \Pot_\fin(E_r)}$ such that $p_F \le \tilde{p}_F$ for each $F$, i.e.,
\[   p_F  \sup_{x\in M} \inf_{y \in F} \abs {x- y}  \le \veps \car
\qquad (F\subseteq E_r\,\, \text{finite}).
\]
Let $q_n := \bigvee \{p_F \suchthat F\subseteq E_r,\, \# F = n\}$ for $n\in \N$. Then $(q_n)_n$ is a partition of unity in $\B$. 
For fixed $n\in \N$ and $F\subseteq E_r$ with $\#F= n$ we 
write $F = \{ y_1^F, \dots, y_n^F\}$. Define the (bounded!) mixings
\[ y^n_j := \sum_{\#F=n} p_F y_j^F + \sum_{\#G\neq n} p_G \, 0\quad (j=1, \dots, n).
\]
Now, let $F_n := \{ y_1^n ,\dots, y_n^n\}$. Then if 
$F\subseteq E_r$ with
$\#F= n$ 
one has, for $x\in M$, 
\[ p_F \inf_{y \in F_n} \abs{x-y} = p_F \inf_{j=1, \dots, n} \abs{x - y_j^n} = p_F \inf_{j=1,\dots, n}\abs{x- y_j^F} = p_F \inf_{y \in F}\abs{x-y} \le \veps\car. 
\]
This yields $q_n \inf_{y \in F_n} \abs{x-y} \le \veps\car$. But
from here it is easy to find $z_n \in \mix(F_n)$ with  
$q_n \abs{x- z} \le \veps \car$.

\prfnoi
(ii)$\dann$(i): It suffices to show that $\mix(M)$ is totally order-bounded.
As $M$ is bounded, so is $\mix(M)$. As in the proof of
the implication (i)$\dann$(ii) we conclude that $\mix(M)$ is
mix-complete. Hence as above we may suppose without loss of generality
that $M$ is mix-complete, and apply Lemma  \ref{rcc.l.rcc-char}.

Fix $\veps \in \R_{>0}$ and pick
$(q_n)_n$ and $(F_n)_n$ as in Lemma \ref{rcc.l.rcc-char}.
Now fix  $n \in \N$ and
write $F_n = \{ y_1, \dots, y_d\}$. For each given 
$x\in M$ we then find
a partition of unity $(p_j)_{j=1}^d$ such that the mixing
$z_n := \sum_{j=1}^d p_j y_j$ satisfies $q_n \abs{x - z_n}\le
\veps\car$. The mapping $z \mapsto \abs{x - z}$ is a $\B$-set map, hence it follows that 
\[   q_n \inf_{y \in F_n} \abs{x-y} \le  \sum_{j=1}^d p_jq_n \abs{x-
  y_j}= q_n \abs{x- z_n} \le \veps\car.
\]
This implies, since $M$ is bounded,  $q_n \sup_{x\in M} \inf_{y \in F_n} \abs {x- y} 
\le \veps\car$ and hence 
\[ q_n \inf_{F\in \Pot_\fin(E)} \sup_{x\in M} \inf_{y \in F} \abs {x- y} \le \veps\car \qquad \text{for each $n \in \N$.}
\]
Since $\bigvee_n q_n = \car$, we obtain 
\[  \inf_{F\in \Pot_\fin(E)} \sup_{x\in M} \inf_{y \in F_n} \abs {x- y} \le \veps\car,
\]
and this implies (i) as $\veps > 0$ was arbitrary.
\end{proof}

\appendix

\section{Background}\label{s.app}

In this appendix we collect some basic notions from the theory of
lattice-normed spaces, see 
\cite[Sec.\,1.2]{EHK2024}, in particular 
\cite[Def.\,1.5 and Rem.\,1.6.(1)]{EHK2024}.

\medskip

\noindent
Let $\A$ be a commutative unital $C^*$-algebra. 
A {\emdf lattice-normed space} over $\A$ is a $\C$-vector space $E$ together with a
 mapping  $\abs{\cdot}\colon E \to \A_+$ with the following properties:
\[  \abs{x}=0 \gdw x=0,\quad \abs{\lambda x} = \abs{\lambda} \abs{x},\quad \abs{x+y}\le \abs{x} + \abs{y} \qquad  (x,y\in E,\, \lambda\in \C).
\]
Each lattice-normed space carries a
natural norm given by $\norm{x}_E := \norm{ \abs{x} }_\A$ for $x\in
E$. 

The algebra  $\A$ itself is a lattice-normed space over itself,
with
$\abs{\cdot}$ being the usual modulus mapping. The induced norm is the
natural one. 

A lattice-normed space over $\A = \C\car$
is nothing but an ordinary normed space.  

\medskip

A {\emdf lattice-normed module} is a lattice-normed space together
with a bilinear mapping 
\[ \A \times E \to E, \qquad (\lambda, x) \mapsto \lambda x
\]
that turns $E$ into an $\A$-module such that 
\[  \abs{ \lambda f} = \abs{\lambda} \abs{f}\qquad (\lambda \in \A,\,
f\in E).
\]

A {\emdf pre-Hilbert lattice-normed module} 
is a lattice-normed module $E$ together  with 
an $\A$-sesquilinear mapping
\[  E\times E \to \A, \qquad (x,y) \mapsto \sprod{x}{y}
\]
satisfying $\sprod{x}{x} = \abs{x}^2$ for all $x\in E$.

\medskip

A subset $M\subseteq E$ is {\emdf order-bounded} if
there is $f\in \A_+$ such that $\abs{x}\le f$ for all $x\in
M$. Clearly, $M$ is order-bounded iff it is norm-bounded.

\medskip

A net $(u_i)_i$ in $\A_+$ {\emdf decreases to $0$} (symbolically: $u_i
\searrow 0$) 
if $(u_i)_i$ is decreasing with $\inf_i u_i = 0$.

A net
$(x_\alpha)_\alpha$ is {\emdf order-convergent} to $x\in E$, if 
there is a net $(u_i)_i$ in $\A_+$ decreasing to $0$ and with the
property 
\[  \forall\, i \,\,\exists \alpha_i \,\, \forall\, \alpha \ge \alpha_i:\,  \abs{x_\alpha - x}\le u_i.
\]
In this case, $x$ is the {\emdf order-limit} of $(x_\alpha)_\alpha$
and as such is uniquely determined. 

If a net $(x_\alpha)_\alpha$ converges to some $x\in E$,
then it order-converges to $x$. 

\medskip
A mapping $f\colon E \to F$ between lattice-normed spaces $E$ and $F$ is
{\emdf order-continuous} if it preserves order-convergence, i.e.:
whenever some (equiv: some bounded) net $(x_\alpha)_\alpha$ order-converges to some $x$ in $E$, then
$Tx_\alpha$ order-converges to $Tx$. 

The vector space operations and the modulus
mapping $\abs{\cdot}\colon E \to \A$ are order-continuous. 

An order-continuous linear mapping is
norm-continuous, and hence maps bounded sets to bounded sets.

\medskip

A subset $M \subseteq E$ is {\emdf order-closed} if it contains the 
order-limit of each (equivalently: each bounded) net in $M$ which is order-convergent (in $E$). 
The {\emdf order-closure} $\ocl(M)$ is the smallest order-closed
subset of $E$ containing $M$, i.e., 
\[ \ocl(M) = \bigcap \{ N \subseteq E \suchthat M \subseteq N,\,N\, 
\text{order-closed}\}.
\]
One can show that
\begin{equation}\label{app.eq.ocl}
  \ocl(M) \subseteq \{ x\in E \suchthat  \inf_{z\in M} \abs{x-z} =0 \}
\end{equation}
as the right-hand side set is order-closed, see \cite[proof of Lemma
1.13]{EHK2024}. 

\begin{lemma}\label{app.l.ocl-prod}
Let $E,F$ be lattice-normed spaces over $\A$ and let $M \subset E$ and
$N \subseteq F$ be subsets. Then 
\[  \ocl(M \times N) = \ocl(M) \times \ocl(N)
\]
in the lattice-normed space $E\times F$,  with  $\abs{(x,y)}_1 := \abs{x}
+ \abs{y}$ for all $x\in E$ and $y\in F$. 
\end{lemma}

\begin{proof}
It is clear that $E\times F$ is lattice-normed over $\A$ with respect
to $\abs{\cdot}_1$. A net $((x_\alpha, y_\alpha))_\alpha$ in
$E\times F$
order-converges to $(x,y)\in E\times F$ iff $x_\alpha \to x$ and
$y_\alpha \to y$ in order. Hence $\ocl(M) \times \ocl(N)$ is
order-closed. If $A\subseteq E\times F$ is order-closed
with $M \times N \subseteq A$, then for each $x\in M$ the 
set $\{y\in F \suchthat (x,y) \in A\}$ is order-closed and contains
$N$, and hence contains $\ocl(N)$. This yields $M \times \ocl(N)
\subseteq Q$. By symmetry, $\ocl(M) \times \ocl(N) \subseteq A$ and
this concludes the proof.  
\end{proof}

A net $(x_\alpha)_\alpha $ is {\emdf order-Cauchy} if 
the net $(\abs{x_\alpha - x_\beta})_{(\alpha,\beta)}$ order-converges
  to $0$. A lattice-normed space $E$ is {\emdf order-complete}
if each (equivalently: each bounded) order-Cauchy net in $E$ is order-convergent.  

Each order-Cauchy net ist norm-Cauchy, and
each order-complete space is norm-complete.

\medskip
A commutative unital $C^*$-algebra $\A$ is a {\emdf
Stone algebra} if it is order-complete (as a lattice-normed space over
itself). Equivalently, $\A$ is Dedekind-complete
as a Banach lattice, see  
\cite[Sec.\,1.3]{EHK2024}. 

By Gelfand's theorem, we may
identify  $\A \cong \Ce(\Om)$, 
where $\Om$ is a compact Hausdorff space. Then
$\A$ is a Stone algebra iff $\Om$ is 
extremally disconnected.

If $E$ is a lattice-normed space over a Stone algebra $\A$, a bounded
net $(x_\alpha)_\alpha$ in $E$
order-converges to some $x\in E$ iff
there is a net $(u_\alpha)_\alpha$ (same
index set!) in $\A_+$ decreasing to $0$ with 
$\abs{x-x_\alpha}\le u_\alpha$ for all $\alpha$,
see also \cite[Lemma 1.10]{EHK2024}.

\medskip

A {\emdf Kaplansky--Hilbert module}
is an order-complete pre-Hilbert lattice-normed
module over a Stone algebra, see
\cite[Def.\,2.1]{EHK2024}.


\bigskip


\parindent 0pt
\parskip 0.5\baselineskip
\setlength{\footskip}{4ex}
\bibliographystyle{alpha}


\end{document}